\documentclass[11pt,a4paper,reqno]{amsart}
 \usepackage{amsaddr}
\usepackage{amscd,amsmath}
\usepackage{amssymb,amsmath,latexsym,color,enumitem,tikz,dsfont,tikz-cd}
\usepackage[all]{xypic}       
\usepackage[varg]{pxfonts}
\usepackage{tabu,adjustbox,booktabs}
\usepackage{graphicx}

\usepackage{hyperref}

\usetikzlibrary{decorations.pathmorphing,shapes}

\usetikzlibrary{decorations.markings}
\tikzset{negated/.style={
        decoration={markings,
            mark= at position 0.5 with {
                \node[transform shape] (tempnode) {$\backslash$};
            }
        },
        postaction={decorate}
    }
}

\newcommand{\donotbreakdash}[1]{#1\nobreakdash-\hspace{0pt}}

 \newcommand{\newdownarrow}{{{\rlap{$\ $}\hbox{$\downarrow$}}}}
 \newcommand{\newuparrow}{{{\rlap{$\ $}\hbox{$\uparrow$}}}}
 \newcommand{\twoheaddownarrow}{{\rlap{\rlap{$\ $}\raise .25ex\hbox{$\downarrow$}}\raise-.25ex\hbox{$\downarrow$}}}
 \newcommand{\twoheaduparrow}{{\rlap{\rlap{$\ $}\raise .25ex\hbox{$\uparrow$}}\raise-.25ex\hbox{$\uparrow$}}}

\newcommand{\set}[1]{\{\,#1\,\}}

\newcommand{\tbigwedge}{\mathop{\textstyle \bigwedge }}
\newcommand{\tbigcap}{\mathop{\textstyle \bigcap }}
\newcommand{\tbigcup}{\mathop{\textstyle \bigcup }}
\newcommand{\tbigoplus}{\mathop{\textstyle \bigoplus }}

\newcommand{\tbigvee}{\mathop{\textstyle \bigvee }}
 \newcommand{\cat}[1]{\ensuremath{\mathsf{#1}}} 
 
\makeatletter
\newcommand*{\@old@slash}{}\let\@old@slash\slash
\def\slash{\relax\ifmmode\delimiter"502F30E\mathopen{}\else\@old@slash\fi}
\makeatother

\def\R{\mathbb{R}}

\def\N{\mathbb{N}}

\def\SS{\mathsf{S}}
\def\CC{\mathsf{C}}
\def\FF{\mathcal{F}}

\def\cl{\mathfrak{c}}
\def\op{\mathfrak{o}}
\def\bl{\mathfrak{b}}
\newtheorem{theorem}{Theorem}[section]
\newtheorem{proposition}[theorem]{Proposition}
\newtheorem{construction}[theorem]{Construction}
\newtheorem{lemma}[theorem]{Lemma}
\newtheorem{corollary}[theorem]{Corollary}

\theoremstyle{definition}
\newtheorem{definition}[theorem]{Definition}
\newtheorem{example}[theorem]{Example}

\theoremstyle{remark}
\newtheorem{remark}[theorem]{Remark}
\newtheorem{remarks}[theorem]{Remarks}

\title[Localic separation and the duality between closedness and fittedness]{Localic separation and the duality between  closedness and fittedness}

\author{Igor Arrieta}
\address{School of Computer Science, \\
University of Birmingham, B15 2TT, Birmingham, UK\\[5mm]
\footnotesize{Dedicated to Professor Jorge Picado on the occasion of his 60th birthday}}
\email{i.arrietatorres@bham.ac.uk}
\thanks{The author acknowledges support from the Basque Government (grant IT1483-22 and a postdoctoral fellowship of the Basque Government, grant POS-2022-1-0015).}

\begin{document}

\maketitle

\begin{abstract}
There are a number of localic separation axioms which are roughly analogous to the $T_1$-axiom from classical topology. For instance, besides the well-known subfitness and fitness, there are also Rosick\' y--\v Smarda's $T_1$-locales, totally unordered locales and, more categorically, the recently introduced $\mathcal{F}$-separated locales  (i.e., those with a fitted diagonal) --- a property strictly weaker than fitness.

It has recently been shown that the strong Hausdorff property and $\mathcal{F}$-separatedness are in a certain sense dual to each other. In this paper, we provide further instances of this duality --- e.g., we introduce a new first-order separation property which is to $\mathcal{F}$-separatedness as the Johnstone--Sun-shu-Hao--Paseka--\v Smarda conservative Hausdorff axiom is to the strong Hausdorff property, and which can be of independent interest. Using this, we tie up the loose ends of the theory by establishing all the possible implications between these properties and other $T_1$-type axioms occurring in the literature.  In particular, we show that the strong Hausdorff property does not imply $\mathcal{F}$-separatedness, a question which remained open and shows a remarkable difference with its counterpart in the category of topological spaces.
\end{abstract}

\section{Introduction}

Let $X$ be a $T_0$-space. It is an elementary fact in general topology that any of the following conditions is equivalent to the $T_1$-axiom:

\begin{enumerate}[label=\textup{(\arabic*)},leftmargin=3.0\parindent]
\item\label{T1.1} Every open set is a union of closed subspaces;
\item\label{T1.2}  Every closed set is saturated (i.e., an intersection of open sets);
\item\label{T1.3}  The diagonal is saturated in $X\times X$;
\item\label{T1.4}  The order in the homset $\cat{Top}(Y,X)$ is discrete for any space $Y$ (recall that $\cat{Top}(Y,X)$ is ordered under the pointwise specialization order; equivalently $f\leq g$ iff $f^{-1}(U)\subseteq g^{-1}(U)$ for any open set $U$ of $Y$);
\item\label{T1.5}  Every point is closed;
\item\label{T1.6}  Every point is saturated.
\end{enumerate}

Now, if one  represents conditions \ref{T1.1}--\ref{T1.6} by  a conceptual analogy  in the category of locales, the resulting axioms are no longer equivalent to each other and lead to distinct conditions which are roughly analogous to the $T_1$-axiom for spaces:

\begin{enumerate}[label=\textup{(\arabic*')},leftmargin=3.0\parindent]
\item\label{T1.1b} Every open sublocale of $L$ is a join of closed sublocales;
\item\label{T1.2b}  Every closed sublocale of $L$ is fitted (i.e., an intersection of open sublocales);
\item\label{T1.3b}  The diagonal is fitted in $L\oplus L$;
\item\label{T1.4b}  The order in $\cat{Frm}(L,M)$ is discrete;
\item\label{T1.5b}  Every one-point sublocale of $L$ is closed;
\item\label{T1.6b}  Every one-point sublocale of $L$ is fitted.
\end{enumerate}

Conditions \ref{T1.1b} and \ref{T1.2b} are the important and well-known \emph{subfitness} and \emph{fitness}, respectively; whereas condition \ref{T1.4b} amounts to Isbell's \emph{unorderedness} \cite{isbellTU} (also called \emph{total unorderedness} in Johnstone, see \cite{STONE}). Conditions \ref{T1.5b} and \ref{T1.6b}  are very point dependent and thus unlikely   to be of much use in point-free topology (but it is important to point out that locales satisfying \ref{T1.5b} were introduced as \emph{$T_1$-locales} by Rosick\' y
 and \v Smarda \cite{smarda}, and have some desirable properties categorically; also, both are implied by several other genuinely point-free assumptions).  Finally, we studied condition \ref{T1.3b} recently under the name \emph{$\FF$-separatedness} in \cite{APP21}, where it was shown that it is in a strong structural parallel with Isbell's strong Hausdorff property (see also \cite{ArrietaThesis} for details, or Subsection~\ref{diagonal.sep} below for an overview).

 However, in view of the fact that the Hausdorff axiom for spaces implies the $T_1$-property; it is a natural question whether it is the case that the strong Hausdorff property also implies $\FF$-separatedness for locales; this was left as an open question when writing \cite{APP21} and it is the main goal of this paper to provide a negative answer;  a consequence of this is that the only implications that hold between properties  \ref{T1.1b}---\ref{T1.6b} are the known ones. We also note that this strenghtens the idea (cf. \cite{APP21}) that the strong Hausdorff property and $\FF$-separatedness are dual to each other: they are not comparable but there are a number of parallel pairs of results for each one; thus showing a remarkable difference with the category  of topological spaces. In this parallel,  total unorderedness seems to play a symmetric role with respect to $\FF$-separatedness  and the strong Hausdorff property  (cf. Figure~\ref{figtotal}).
 
 In this context, we note that such a counterexample yields \emph{a fortiori} an example of a strongly Hausdorff locale which is not fit, but these are already rather hard to find --- with regard to the historical background, Isbell first constructed such an example in \cite[2.3~(4')]{atomless};  it was later generalized by Banaschewski  to the class of simple extensions of regular spaces (under some mild conditions) in \cite{singly}. However, all such examples are $\mathcal{F}$-separated as we proved in \cite{APP21}; hence they do not serve for our purpose herein.
 
In this paper, building on a construction due to Juhász et al. \cite{juhasz} we construct strongly Hausdorff locales which fail to be $\mathcal{F}$-separated; in particular providing new examples of  strongly Hausdorff locales  which are not fit.
The main tool for this purpose is a new first-order separation axiom for locales ---called property \ref{proper.F} in what follows---  introduced by the author in his PhD thesis \cite{ArrietaThesis}. We show that property \ref{proper.F} is the dual counterpart of the conservative Hausdorff axiom due to Johnstone--Sun-shu-Hao--Paseka--\v Smarda (cf. \cite{johnstoneHao,paseka,separation}).

We then put  these new separation axioms into perspective by comparing and relating them with  \ref{T1.1b}--\ref{T1.6b} above, and also with other  similar properties which appear in the literature --- notably weak subfitness and prefitness  \cite{subfit2015}.

\section{Preliminaries} 
Our notation and terminology regarding the categories of frames and locales will be that of \cite{PP12}; and for topics related to localic separation we refer to \cite{separation}. The Heyting operator in a frame $L$, right adjoint to the meet operator, will be denoted by $\to$; for each $a\in L$, $a^*=a\to 0$ is the pseudocomplement of $a$. 

\subsection{Some Heyting rules}
For the reader's convenience, we list here some of the properties satisfied by the Heyting operator in a frame $L$. For any $a,b,c\in L$, the following hold:
\begin{enumerate}[label=\textup{(H\arabic*)},leftmargin=2.0\parindent]
\item \label{H1} $1\to a=a$\textup;
\item \label{H2} $a\leq b$ if and only if $a\to b=1$\textup;
\item \label{H3} $a\leq b\to a$\textup;
\item \label{H4} $a\to b=a\to (a\wedge b)$\textup;
\item \label{H5} $a\wedge (a\to b)=a\wedge b$\textup;
\item \label{H6} $a\wedge b=a\wedge c$ if and only if $a\to b=a\to c$\textup;
\item \label{H7} $(a\wedge b)\to c=a\to (b\to c)= b\to (a\to c)$\textup;
\item \label{H8} $a=(a\vee b)\wedge (b\to a)$\textup;
\item \label{H9} $a\leq (a\to b)\to b$\textup;
\item \label{H10} $((a\to b)\to b)\to b=a\to b$\textup.
\end{enumerate}

\subsection{Sublocales}\label{subsect.sublocales}

A \emph{sublocale} of a locale $L$ is a subset $S\subseteq L$ closed under arbitrary meets such that
\[\forall a\in L,\ \ \ \forall s\in S,\  \ \ a\to s\in S.\]
These are precisely the subsets of $L$ for which the embedding $j_S\colon S\hookrightarrow L$ is a morphism of locales. Sublocales of $L$ are in one-to-one correspondence with the regular subobjects (equivalently, extremal subobjects) of $L$ in \cat{Loc}. If $\nu_S$ denotes the associated frame surjection, then for any $a\in L$ and $s\in S$ one has
\begin{equation}\label{LocM}\tag{LM}
\nu_S(a)\to s=a\to s.\end{equation}

The system $\SS  (L)$ of all sublocales of $L$, partially ordered by inclusion, is a coframe \cite[Theorem~III.3.2.1]{PP12}, that is, its dual lattice is a frame.  Infima and suprema are given by
\[
\tbigwedge_{i\in I}S_i=\tbigcap_{i\in I}S_i, \quad \tbigvee_{i\in I}S_i=\set{\tbigwedge M\mid M\subseteq\tbigcup_{i\in I} S_i}.
\]
The least element is the sublocale $\mathsf{O}=\{1\}$ and the greatest element is the entire locale $L$. For any $a\in L$, the sublocales
\[
\mathfrak{c}_L(a)=\newuparrow  a=\set{b\in L\mid b\ge a}\ \text{ and }\ \mathfrak{o}_L(a)=\set{a\to b\mid b\in L}
\]
are the \emph{closed} and \emph{open} sublocales of $L$, respectively (that we shall denote simply by $\mathfrak{c}(a)$ and $\mathfrak{o}(a)$ when there is no danger of confusion). For each $a\in L$, $\mathfrak{c}(a)$ and $\mathfrak{o}(a)$ are
complements of each other in $\SS(L)$
and satisfy the expected identities
\begin{equation*}{\label{identities.basic}}
\tbigcap_{i\in I} \mathfrak{c}(a_i)=\cl(\tbigvee_{i\in I} a_i),\quad \cl(a)\vee\cl(b)=\cl(a\wedge b),
\end{equation*}
\[\tbigvee_{i\in I}\op(a_i)=\op(\tbigvee_{i\in I} a_i) \quad\mbox{ and }\quad \op(a)\cap \op(b)=\op(a\wedge b).
\]

Given a sublocale $S$ of $L$, its \emph{closure}, denoted by $\overline{S}$, is the smallest closed sublocale containing it. In this context, the formula $\overline{S}=\cl(\tbigwedge S)$ holds.

A sublocale is said to be \emph{fitted} if it is the intersection of all the open sublocales containing it. Hence, it constitutes a (non-conservative) point-free extension of the notion of saturated subspace. If $S$ is a sublocale  of $L$, the \emph{fitting} of $S$ is the intersection of the open sublocales containing $S$, that is, the smallest fitted sublocale containing $S$ (cf. \cite{otherclosure}).

Finally, if $p$ is a prime element in $L$, the subset $\bl(p)=\set{1,p}$ is easily seen to be a sublocale, sometimes referred to as a \emph{one-point sublocale}.

\subsection{Some standard separation axioms}\label{subsection.standard.sep}

We now recall some important point-free separation axioms (for a comprehensive account of the topic we refer to \cite{separation}).  A frame $L$ is said to be
\begin{enumerate}[label       = \textbullet,
                  labelindent = 0\parindent,
                  leftmargin  = \parindent,
                  labelsep    = *,
                  topsep      = 10pt,
                  itemsep      = 5pt]\item \emph{regular} if for any $a\in L$, the relation $a=\tbigvee\set{ b\mid b\prec a}$ holds, where $b\prec a$ means that $b^*\vee a=1$.
\item \emph{fit} if for any $a,b\in L$ with $a\not\leq b$, there exists a $c\in L$ such that $a\vee c=1$ and $c\to b\not\leq b$.
\item \emph{totally unordered} (or that it satisfies \ref{proper.TU})  if for any pair of frame homomorphisms $h,k\colon L\to M$, 
\begin{equation}{\label{proper.TU}}
\text{the relation } h\leq k \text{ implies } h=k. \tag*{\textup{(}$T_U$\textup{)}}
\end{equation}
\item \emph{$T_1$} if every prime in $L$ is maximal, that is
\begin{equation}{\label{proper.T1}}
\text{if for any prime } p\in L \text{ and any } a>p,\text{ then } a=1.\tag*{\textup{(}$T_1$\textup{)}}
\end{equation}
\item \emph{subfit} if for any $a,b\in L$ with $a\not\leq b$, there  exists a $c\in L$ such that $a\vee c=1\ne b\vee c$.
\item \emph{weakly subfit} if for any $a>0$ in $L$, there exists a $c<1$ with $a\vee c=1$.
\item \emph{prefit} if for any $a>0$ in $L$, there exists a $c>0$ such that $c\prec a$ (i.e., $a\vee c^*=1$).

\end{enumerate}

Although the definitions of subfitness and fitness  are given as first-order properties, as mentioned in the Introduction they can be characterized geometrically: a locale is fit if and only if each of its closed sublocales is fitted (equivalently, iff \emph{any} sublocale whatsover is fitted), and a locale is subfit if and only if each open sublocale is a join of closed sublocales (cf. \cite{atomless,PP12,separation}). Furthermore, a locale is $T_1$ if and only if for each prime $p\in L$, the sublocale $\bl(p)$ is closed (cf. \cite{smarda}).

With regard to the relations between these properties, the implications
$$\text{Regular}\implies\text{Fit}\implies\text{Subfit}\implies\text{Weakly subfit}$$
are easily seen to hold and are all strict; for the remaining relations we refer to \cite{separation} or to Section~\ref{summarlocl} below.
\subsection{Products of locales} We shortly describe a construction of binary coproducts of frames (that is, binary products of locales). For more information, we refer to \cite[IV~4--5]{PP12}. Let $L_1$ and $L_2$  be frames and denote by $\mathcal{D}(L_1\times L_2)$ the frame of downsets of $L_1\times L_2$. A downset $D\in \mathcal{D}(L_1\times L_2)$ is said to be a \emph{cp-ideal} if for all $\{a_i\}_{i\in I}\subseteq L$, $a\in L$, $\{b_j\}_{j\in J}\subseteq M$ and  $b\in M$, 
\begin{align*}
 &(a_i,b)\in D\quad \text{for all }i\in I\implies \bigl(\tbigvee_{i\in I}a_i,b\bigr)\in D,\quad\text{and}\\
 &(a,b_j)\in D\quad \text{for all }j\in J\implies \bigl(a,\tbigvee_{j\in J}b_j\bigr)\in D.
\end{align*}
We note that for any $(a,b)\in L_1\times L_2$, the set 
$$a\oplus b:= \newdownarrow(a,b)\cup \set{(x,y)\mid x=0\text { or } y=0}$$ is the smallest cp-ideal containing $(a,b)$.
Since intersections of cp-ideals are cp-ideals, the set
$$L_1\oplus L_2:=\set{ D\in\mathcal{D}(L_1\times L_2)  \mid  D \text { is a cp-ideal}} $$
is a complete lattice. In fact, it can be shown that $L_1\oplus L_2$ is a frame, and together with the frame homomorphisms $\iota_i\colon L_i\to L_1\oplus L_2$ given by 
$$\iota_{1}(a)=a\oplus 1,\qquad \iota_2(b)=1\oplus b$$
the system $(L_1\oplus L_2,\iota_1,\iota_2)$ is the coproduct of $L_1$ and $L_2$ in the category of frames. Equivalently, $(L_1\oplus L_2, \pi_1,\pi_2)$ in the product of $L_1$ and $L_2$ in the category of locales; where the localic map $\pi_i$ is the right adjoint of $\iota_i$ for $i=1,2$ --- i.e., for any $D\in L_1\oplus L_2$ one has
$$\pi_1(D)=\tbigvee \set{ a\in L_1\mid (a,1)\in D}\quad \text{and}\quad \pi_2(D)=\tbigvee \set{ b\in L_2\mid (1,b)\in D}.$$
Given localic maps $f_i\colon M\to L_i$ for $i=1,2$, the induced map $(f,g)\colon M\to L_1\oplus L_2$ is readily seen to be given by 
$$(f,g)(c)=\set{ (a,b)\in L_1\times L_2\mid f^*(a)\wedge g^*(b)\leq c}.$$

In particular, if $L$ is a locale, the diagonal localic map  $( 1_L,1_L ) \colon L\to L\oplus L$ is injective and so its image defines a sublocale of $L$, namely  $$D_L:=( 1_L,1_L) [L] \subseteq L\oplus L,$$
where $( 1_L,1_L )(a)=\set{(u,v)\in L\times L\mid u\wedge v\leq a}$. We refer to $D_L$  as the \emph{diagonal sublocale} (see \cite[IV~5.3]{PP12}). In particular, for a $D\in D_L$ one has 
\begin{equation}(a,b)\in D \text{  if and only if } (b,a)\in D.\tag{Sym}\label{Sym}\end{equation}

 In this context, we point out that the closure of the diagonal in $L\oplus L$ is given by $\overline{D_L}=\cl (d_L)$, where $$d_L=\tbigvee \set{ a\oplus b\mid a\wedge b=0}=\set{(a,b)\in L_1\times L_2\mid a\wedge b=0}$$ 
(see \cite[V~2.1]{PP12} for details).
\

\subsection{Diagonal separation in \cat{Loc}}\label{diagonal.sep}
Let $\CC$ be a category equipped with a  closure operator $c$ with respect to a proper factorization system $(\mathcal{E},\mathcal{M})$ (see \cite{closure}, cf. also \cite{clementino0,clementino01,clementino1} and references therein). An object $X$ of $\CC$ is said to be \emph{$c$-separated} if the diagonal subobject $(1_X,1_X)\colon X\hookrightarrow X\times X$ is $c$-closed. 

In $\CC=\cat{Top}$, the category of topological spaces, important examples are $c=k$ (the usual topological closure) or $c=s$ (the closure operator sending a subspace to its \emph{saturation} --- i.e., the intersection of all the open subspaces containing it).  As is well known, the $k$-separated objects are precisely the Hausdorff spaces, and the $s$-separated objects are precisely the $T_1$-spaces. 

In $\CC=\cat{Loc}$, the category of locales, the analogous closure operators are given by $c=\mathcal{K}$ (the usual localic closure, recall Subsection~\ref{subsect.sublocales}) and $c=\mathcal{F}$ (the closure operator sending a sublocale to its fitting, see Subsection~\ref{subsect.sublocales}). 

The resulting notions of separation are as follows. On the one hand $\mathcal{K}$-separated locales are precisely the \emph{strongly Hausdorff} locales introduced by Isbell \cite{atomless} (that is, locales whose diagonal is closed) and, on the other hand,  the \emph{$\mathcal{F}$-separated} locales (that is, those whose diagonal is fitted) were  studied recently in \cite{APP21}. 

We point out that both strongly Hausdorff locales and $\FF$-separated locales have an excellent categorical behaviour. Indeed, it follows from general results of the theory of closure operators (see e.g. \cite[Proposition~4.2]{clementino01}) that the following properties hold:

\begin{proposition}\label{categprop.F.sep}
The following assertions hold\textup:
 \begin{enumerate}[label=\textup{(\arabic*)},leftmargin=1.5\parindent]
 \item \label{categprop.F.sep4} If $f,g\colon M\to L$ are localic maps and $L$ is strongly Hausdorff (resp. \donotbreakdash{$\mathcal{F}$}separated), then their equalizer $\mathsf{equ}(f,g)$ is a closed (resp. fitted) sublocale of $M$\textup;
\item \label{categprop.F.sep1} Strongly Hausdorff (resp. \donotbreakdash{$\mathcal{F}$}separated) locales are closed under mono-sources in \cat{Loc}. In particular, if $f\colon M\to L$ is a  monomorphism in \cat{Loc} and $L$ is strongly Hausdorff (resp. \donotbreakdash{$\mathcal{F}$}separated), then so is M\textup;
\item \label{categprop.F.sep2} Strongly Hausdorff (resp. \donotbreakdash{$\mathcal{F}$}separated) locales are closed under limits  in \cat{Loc}\textup;
\item \label{categprop.F.sep3}Strongly Hausdorff (resp. \donotbreakdash{$\mathcal{F}$}separated) locales are extremally epireflective in \cat{Loc}\textup.
\end{enumerate}
\end{proposition}

\begin{remark}
We emphasize that the closure under monomorphisms  in Proposition~\ref{categprop.F.sep}\,\ref{categprop.F.sep1} is significantly more general than closure under sublocales (i.e., regular monomorphisms). Indeed, in the category \cat{Loc} the structure of monomorphisms is fairly wild, and thus this property may be somewhat surprising.
\end{remark}

Recall the fitness property from Subsection~\ref{subsection.standard.sep}. It is well known (cf. \cite{atomless}) that fitness is closed under products. It immediately follows that it implies $\FF$-separatedness. The reverse implication does not hold in general, as it was proved in \cite{APP21}:

\begin{theorem}[{\cite[Thm.~6.4]{APP21}}]\label{fitnessimpliesfsep}
Fitness implies $\FF$-separatedness, but the converse does not hold in general.
\end{theorem}

In \cite{APP21} (see also \cite{ArrietaThesis} for a more detailed account, cf. also Table~\ref{table.resumen}) it was shown that there is a pleasant parallel between the strong Hausdorff property and  $\mathcal{F}$-separatedness.  For instance, both can be characterized by a Dowker-Strauss type condition on the combinatorial structure of the frame homomorphisms with a given domain (cf. \cite[Sect.~4]{APP21}).

Moreover, both properties can be decomposed as the conjunction of total unorderedness  (cf. Subsection~\ref{subsection.standard.sep}) and a certain condition involving weakened frame homomorphisms (cf. \cite[Sect.~5]{weakened} for the closed case and \cite[Sect.~6]{APP21} for the fitted case). 
In particular, one has the following:

\begin{theorem}[{\cite[Cor.~4.5.1]{APP21}}]\label{fsep.imp.tu}
$\FF$-separatedness implies total unorderedness.
\end{theorem}

\begin{table}
\caption{The closed-fitted duality}
{\label{table.resumen}}
\begin{center}
{\small\begin{tabular}{p{0.25\textwidth}p{0.33\textwidth}p{0.33\textwidth}}
\toprule[1.2pt] {\sc } & {\sc Closedness} & {\sc Fittedness} \\
\bottomrule[1.2pt]
\addlinespace[5pt]
\raggedright  Dowker-Strauss-type characterization & \raggedright Strong Hausdorff $\equiv$  no distinct homomorphisms respect disjoint pairs\index{frame homomorphism! pair of~$\sim$~respecting disjoint pairs |textsl} &\raggedright $\FF$-separated $\equiv$ no distinct homomorphisms\index{frame homomorphism!pair of~$\sim$~respecting covers|textsl} respect covers
\tabularnewline
\addlinespace[5pt]\midrule\addlinespace[5pt]
\raggedright 
Relaxed morphisms& \raggedright Weak homomorphisms:\index{homomorphism!weak~$\sim$|textsl} \\[2pt]
(1) Morphism in $\cat{Sup}$ \\ (2) Preserve $\top$ \\ (3) Preserve disjoint pairs &\raggedright Almost homomorphisms:\index{homomorphism!almost~$\sim$|textsl} \\[2pt]
(1) Morphism in $\cat{PreFrm}$ \\ (2) Preserve $\bot$ \\ (3) Preserve covers
\tabularnewline
\addlinespace[5pt]\midrule\addlinespace[5pt]
\raggedright Every relaxed morphism is a\\ frame homomorphism & \raggedright Property~\textup{(W)}& Property~\textup{(A)}\tabularnewline
\addlinespace[5pt]\midrule\addlinespace[5pt]
\raggedright Sufficient condition & \raggedright Hereditary normality implies property~(W) &\raggedright Hereditary extremal disconnectedness implies property (A)\tabularnewline
\addlinespace[5pt]\midrule\addlinespace[5pt]
\raggedright Downset frames & \raggedright $\mathsf{Dwn}(X)$ is hereditarily normal iff it satisfies (W) & \raggedright  $\mathsf{Dwn}(X)$ is hereditarily extremally disconnected iff it satisfies (A) 
\tabularnewline
\addlinespace[5pt]\midrule\addlinespace[5pt]
\raggedright Characterization by relaxed morphisms & \raggedright Strong Hausdorff $\equiv$  \\ ($T_U$) + (W) & \raggedright $\FF$-separated $\equiv$ ($T_U$) + (A)\tabularnewline
\addlinespace[5pt]\midrule\addlinespace[5pt]
\raggedright Associated first order property\index{frame!with property \textup{(H)}|textsl}\index{property!H@\textup{(H)}|textsl} & \raggedright Property~\ref{proper.H} & Property~\ref{proper.F}\index{frame!with property \textup{(F)}|textsl}\index{property!F@\textup{(F)}|textsl}
\tabularnewline
 
\addlinespace[5pt]\bottomrule[1.2pt]
\end{tabular}
}
\end{center}
\end{table}

\subsection{The conservative Hausdorff property}\label{subs.cons.haus}
Since the functor $\Omega\colon\cat{Top}\to\cat{Loc}$ does not preserves products, there is no reason to assume that the strong Hausdorff property is a conservative extension of its classical counterpart. Indeed, if $X$ is a topological space and $\Omega(X)$ is strongly Hausdorff, then $X$ is Hausdorff; but the converse does not hold in general. In view of this situation, several groups of authors investigated possible point-free conservative extensions of the Hausdorff axiom. From very differently motivated approaches, Paseka and \v Smarda \cite{paseka} and Johnstone and Sun-shu-Hao \cite{johnstoneHao} obtained a solution for this problem   (for a detailed historical account of the subject see also \cite[III.1]{separation}).

A frame is said to be \emph{Hausdorff} (or that it satisfies property \ref{proper.H})  if
\begin{equation}{\label{proper.H}}
1\ne a\not\leq b \ \ \implies\ \ \exists u,v\in L\text{ such that }u\not\leq a,\ v\not\leq b\text{ and } u\wedge v=0. \tag*{\textup{(H)}}
\end{equation}

\begin{proposition}
Every strongly Hausdorff locale is Hausdorff.
\end{proposition}

Moreover, property \ref{proper.H} is a conservative extension of the homonymous topological axiom --- i.e., a $T_0$-topological space $X$ is Hausdorff if and only if the locale $\Omega(X)$ is Hausdorff (the importance of conservativeness should not be overestimated, though: it is the strong Hausdorff property, and not property \ref{proper.H}, that in the presence of compactness behaves as expected, see \cite{separation}).

\section{A dual to the conservative Hausdorff property --- the property (F)}
If we refer to the structural parallel between the strong Hausdorff property and $\mathcal{F}$-separatedness (cf. Subsection~\ref{diagonal.sep}), it is a natural question whether the conservative Hausdorff axiom (cf. Subsection~\ref{subs.cons.haus}) has a natural dual counterpart for $\mathcal{F}$-separatedness --- i.e.,  whether there is a first-order separation-type property which is to  $\mathcal{F}$-separatedness as propety \ref{proper.H} is to the strong Hausdorff property. In this section we show that the property \ref{proper.F}, introduced by the  author in his PhD thesis \cite{ArrietaThesis}  provides an affirmative answer to this question.

\begin{theorem}{\label{fof}}
For a frame $L$, the following conditions are equivalent and are all implied by \donotbreakdash{$\mathcal{F}$}separatedness\textup:
\begin{enumerate}[label=\textup{(\roman*)},leftmargin=1.5\parindent]
\item \label{fof1} For every $a,b\in L$ such that $1\ne a\not\leq b$, there exist $u,v\in L$  such that  $u\not\leq a$, $v\not\leq b$ and  $(u\to a)\vee (v\to b)=1$\textup;
\item \label{fof2} For every $a,b\in L$  such that  $1\ne a\not\leq b$, there exist $u,v\in L$  such that  $a<u$, $b<v$ and  $(u\to a)\vee (v\to b)=1$\textup;
\item \label{fof3} For every $a,b\in L$  such that  $1\ne a\not\leq b$, there exist $u,v\in L$  such that  $v\leq a<u$, $v\not\leq b$ and $(u\to a)\vee (v\to b)=1$\textup;
\item \label{fof4} For every $a,b\in L$  such that  $1\ne a\not\leq b$, there exist $u,v\in L$  such that  $u\to a\ne a$, $v\to b\ne b$ and $u\vee v=1$\textup;
\item \label{fof5} For every $a,b\in L$  such that  $1\ne a\not\leq b$, there exist $u,v\in L$  such that  $a\leq u$, $b\leq v$, $u\to a\ne a$, $v\to b\ne b$ and $u\vee v=1$\textup;
\item \label{fof6} For every $a,b\in L$  such that  $1\ne a\not\leq b$, there exist $u,v\in L$  such that  $a\leq u$, $u\to a\ne a$, $a\wedge (v\to b)\not\leq b$ and $u\vee v=1$\textup.
\end{enumerate}
\end{theorem}

\begin{proof}
Let us start by showing that \donotbreakdash{$\mathcal{F}$}separatedness implies \ref{fof1}. Let $1\ne a\not\leq b$. Then $$a\invamp b=\set{(x,y)\in L\times L\mid x\leq a \textrm{ or } y\leq b}$$ 
is clearly a cp-ideal, and since $(a,1)\in a\invamp b$ and $(1,a)\notin a\invamp b$, it follows from \eqref{Sym} that $a\invamp b\not\in D_L$. Hence $a\invamp b\not\in D_L=\tbigcap_{D_L\subseteq \op(U)}\op(U)$ because $L$ is \donotbreakdash{$\mathcal{F}$}separated and so there exists a $U\in L\oplus L$ such that $D_L\subseteq\op(U)$ and $a\invamp b\notin \op(U)$ --- i.e., $\tbigcap_{(x,y)\in U}((x\oplus y)\to a\invamp b)\not\subseteq a\invamp b$. Therefore, there is a pair $(u,v)\in L\times L$ such that for all $(x,y)\in U$, one has $(u,v)\in (x\oplus y)\to a\invamp b$ but $(u,v)\not\in a\invamp b$. The latter means $u\not\leq a$ and $v\not\leq b$; while the former means that for all $(x,y)\in U$ one has $(u\wedge x)\oplus (v\wedge y)\subseteq a\invamp b$, or equivalently $(u\wedge x,v\wedge y)\in a\invamp b$. Hence, for each $(x,y)\in U$, one has either $u\wedge x\leq a$ or $v\wedge y\leq b$ and so $x\wedge y\leq (u\to a)\vee(v\to b)$. Since $D_L\subseteq \op(U)$, the system $\widehat{U}=\set{x\wedge y\mid (x,y)\in U}$ is easily seen to be a cover of $L$, hence $(u\to a) \vee (v\to b)=1$.

We now check that all the conditions are equivalent:\\[2mm]
\ref{fof1}$\implies$\ref{fof2}: Let $1\ne a\not\leq b$. Then there are $u,v\in L$ with $u\not\leq a$, $v\not\leq b$ and $(u\to a)\vee (v\to b)=1$. Set $u':=u\vee a$ and $v':=v\vee b$. Then $a<u'$, $b<v'$ and $(u'\to a)\vee (v'\to b)=(u\to a)\vee (v\to b)=1$.\\[2mm]
\ref{fof2}$\implies$\ref{fof3}: Let $1\ne a\not\leq b$. Then one has $1\ne a\not\leq a\to b$ and hence there exist $u,v\in L$ such that $a<u$, $a\to b<v$ and $(u\to a)\vee (v\to (a\to b))=1$. Let $v':=v\wedge a$. Since $v\not\leq a\to b$, one has $v'\not\leq b$ (and $v'\leq a$). Moreover, by \ref{H7}, it follows that $v\to (a\to b)=(v\wedge a)\to b=v'\to b$. Hence the pair $u,v'$ satisfies the required conditions.\\[2mm]
\ref{fof3}$\implies$\ref{fof4}: Let $1\ne a\not\leq b$. Then there exist $u,v\in L$ such that $v\leq a<u$, $v\not\leq b$ and $(u\to a)\vee (v\to b)=1$. Let $u':=u\to a$ and $v':= v\to b$. Then $u'\vee v'=1$. Moreover, if $u'\to a\leq a$, then $u\leq (u\to a)\to a\leq a$ by \ref{H9}, a contradiction. Hence $u'\to a\ne a$ and similarly, $v'\to b\ne b$.\\[2mm]
\ref{fof4}$\implies$\ref{fof5} follows easily because $u\to a=(u\vee a)\to a$ and $v\to b=(v\vee b)\to b$; thus we may replace $u$ (resp. $v$) by $u\vee a$ (resp. $v\vee b$).\\[2mm]
\ref{fof5}$\implies$\ref{fof6}: Let $1\ne a\not\leq b$. Then $1\ne a\not\leq a\to b$ and hence there exist $u,v\in L$ such that $a\leq u$, $a\to b\leq v$, $u\to a\ne a$, $v\to (a\to b)\ne a\to b$ and $u\vee v=1$. By \ref{H5} and \ref{H7} one has ${a\wedge (v\to b)}=a\wedge (a\to (v\to b))=a\wedge (v\to (a\to b))\not \le b$.\\[2mm]
\ref{fof6}$\implies$\ref{fof1}: Let $1\ne a\not\leq b$. Then there exist $u,v\in L$ such that $a\leq u$, $u\to a\ne a$, $a\wedge (v\to b)\not\leq b$ and $u\vee v=1$. Let $u':=u\to a$ and $v':=v\to b$. Then $u'\not\leq a$, $v'\not\leq b$ and $(u'\to a) \vee (v'\to b)\geq u\vee v=1$ by \ref{H9}.
 \end{proof}

A frame satisfying one (and hence all) of the equivalent conditions above will be said to satisfy property~\ref{proper.F}, that is, $L$ satisfies \emph{property~\ref{proper.F}} if \index{frame!with property \textup{(F)}|textsl}\index{property!F@\textup{(F)}|textsl}
\begin{equation}{\label{proper.F}}
1\ne a\not\leq b \ \ \implies\ \ \exists u,v\in L\text{ such that }u\to a\ne a,\ v\to b\ne b\text{ and } u\vee v=1. \tag*{\textup{(F)}}
\end{equation}

\begin{corollary}
Fitness implies property \ref{proper.F}
\end{corollary}

\begin{remark}\label{fitness.F}
 We point out here that the implication ($\mathcal{F}$-sep)$\implies$ \ref{proper.F} cannot be reversed. 
Such a counterexample can be  given already in the context of singly generated frame extensions \cite{singly}. Indeed, let $\tau$ be the topology on the real line given by open sets of the form $U\cup (V\cap \mathbb{Q})$ where $U,V\in \tau_{us}$, where $\tau_{us}$ denotes the usual topology on the real line. It is known that $\tau$ is not totally unordered (cf. \cite[Example~III\,1.5]{STONE}); so in particular it is not $\mathcal{F}$-separated (recall Theorem~\ref{fsep.imp.tu}). On the other hand, $\tau$ satisfies property \ref{proper.F}. Indeed, let $U,V\in \tau$ with $\R\ne U\not\subseteq V$. Select an $x\in U$ with $x\not\in V$ and a $y\not\in U$. Since $x\ne y$ and $(\R,\tau_{us})$ is regular, there are open sets $U_1, U_2\in\tau_{us}$ with $x\in U_1$, $y\in U_2$ and whose closures $\overline{U_1}$ and $\overline{U_2}$ in  $(\R,\tau_{us})$ are disjoint. Let $U'=\R-\overline{U_2}$ and $V'=\R-\overline{U_1}$. Then $U',V'\in\tau$ because $\tau_{us}\subseteq \tau$. Clearly, $U'\cup V'=\R$. Moreover, if $\mathsf{int}(U\cup \overline{U_2})=U'\to U\subseteq U$, then since $y\in  U_2 \subseteq \mathsf{int}(U\cup \overline{U_2})$, we would have $y\in U$, a contradiction. Similarly, one verifies that $V'\to V\not\subseteq V$.
\end{remark}

Recall that a locale is $T_1$ if every prime is maximal (cf. Subsection~\ref{subsection.standard.sep}). In this context, we have the following:

\begin{proposition}\label{FimpliesT1}
Property~\ref{proper.F} implies property~\ref{proper.T1}.
\end{proposition}

\begin{proof}
Let $L$ be a frame and let $p\in L$ be a prime. Assume by contradiction that $p$ is not maximal, i.e., $p\leq a\leq 1$ with $a\not\leq p$ and $a\ne 1$. By hypothesis there exist $u,v\in L$ such that $u\to a\ne a$, $v\to p\ne p$ and $u\vee v=1$.
Now, since $p$ is prime, $v\to p\ne p$ implies $v\leq p$ and so it follows that $u\vee p=1$. In particular, $u\vee a=1$, and so $u\to a= (u\vee a)\to a=1\to a= a$, which yields a contradiction.
\end{proof}

In what follows, we investigate the categorical behaviour of property \ref{proper.F}.

\begin{proposition}{\label{Fishered}}\index{frame!with property \textup{(F)}|textsl}\index{property!F@\textup{(F)}|textsl}\index{property!hereditary~$\sim$|textsl}
Property~\ref{proper.F} is hereditary.
\end{proposition}

\begin{proof}
Let $L$ be a frame which has property~\ref{proper.F} and let $S\subseteq L$ be a sublocale with corresponding surjection $\nu_S\colon L\twoheadrightarrow S$. We denote by $\vee^S$ (resp. $\vee$) the join in $S$ (resp. $L$). Let $a,b\in S$ such that $1\ne a\not\leq b$. Since $L$ satisfies property~\ref{proper.F}, there exist $u,v\in L$ such that $u\to a\ne a$, $v\to b\ne b$ and $u\vee v=1$. Let $u'=\nu_S(u)$ and $v'=\nu_S(v)$. Then $u'\vee^S v'=\nu_S(u\vee v)=1$. Moreover, by  \eqref{LocM} one readily checks that $u'\to a\ne a$ and $v'\to b\ne b$, hence $u',v'\in S$ yield the required elements for property \ref{proper.F} in $S$.
\end{proof}

\begin{proposition}{\label{productsF}}\index{frame!with property \textup{(F)}|textsl}\index{property!F@\textup{(F)}|textsl}
Arbitrary products of locales with property~\ref{proper.F} also have property~\ref{proper.F}.
\end{proposition}

\begin{proof}
The first part of the proof follows the same lines of that of \cite[Lemma~1.9]{paseka}, cf. also \cite[p.~45]{separation}. Let $\{L_i\}_{i\in I}$ be a family of frames satisfying property~\ref{proper.F}.
Let $1\ne V\not\subseteq W$ in $\tbigoplus_{i\in I}L_i$. Pick $\pmb{a}=(a_i)_{i\in I}\in V-W$. Let $\set{i_1,\dots,i_n}$ be the set of indices such that $a_{i_j}\ne 1$ for all $j=1,\dots,n$. Let $\pmb{a}^{(0)}:=\pmb{a}$ and for each $j=1,\dots,n$, let $\pmb{a}^{(j)}$ be the element $\pmb{a}$ but with all the entries in $i_1,\dots,i_j$ replaced by $1$. Since $\pmb{a}^{(0)}=\pmb{a}\in V$ and $\pmb{a}^{(n)}=(1)_{i\in I}\not\in V$, there is an $j_0\in\set{1,\dots,n}$ such that $\pmb{a}^{(j_0-1)}\in V$ but $\pmb{a}^{(j_0)}\not\in V$. 

For each $x\in L_{i_{j_0}}$ let $\pmb{x}$ be $\pmb{a}^{(j_0)}$ but with the $1$ in position $i_{j_0}$ replaced by $x$.
Further, let $v:=\tbigvee\set{x\in L_{i_{j_0}}\mid \pmb{x}\in V}$ and $w:=\tbigvee\set{x\in L_{i_{j_0}}\mid \pmb{x}\in W}$. Because $V$ and $W$ are $cp$-ideals, one has $\pmb{v}\in V$ and $\pmb{w}\in W$. If $v=1$, then $\pmb{a}^{(j_0)}=\pmb{v}\in V$, a contradiction. Thus $v\ne 1$. Assume $v\leq w$.
 Now, since $\pmb{a_{i_{j_0}}}=\pmb{a}^{(j_0-1)}\in V$, it follows that $a_{i_{j_0}}\leq v\leq w$, and so $\pmb{a}\leq \pmb{w}\in W$. Since $W$ is a downset, it follows that $\pmb{a}\in W$, a contradiction. Hence $1\ne v\not\leq w$.
 
Since $L_{i_{j_0}}$ satisfies property \ref{proper.F} there are $x,y\in L_{i_{j_0}}$, with $x\to v\not\leq v$, $y\to w\not\leq w$ and $x\vee y=1$.
 
 Let $x_i=1=y_i\in L_i$ for each $i\ne i_{j_0}$ and let $x_{i_{j_0}}=x$, $y_{i_{j_0}}=y$. Then obviously $(\oplus_i x_i)\vee(\oplus_i y_i)=1$.
 
We claim that $(\oplus_i x_i)\to V\ne V$. Assume otherwise, by contradiction. Since $x\to v\not\leq v$, there is a $c\in L_{i_{j_0}}$ such that $c\leq x\to v$ (i.e., $c\wedge x\leq v$) and $c\not\leq v$.
 One obviously has $\pmb{c}\wedge (x_i)_i = \pmb{c\wedge x}\leq \pmb{v}\in V$ and since $V$ is a downset, we deduce that $\pmb{c}\wedge (x_i)_i\in V $. It follows that $\pmb{c}\in (\oplus_i x_i)\to V= V$. But if $\pmb{c}\in V$, one has by definition of $v$ that $c\leq v$, a contradiction. The fact that $(\oplus y_i)\to W\ne W$ may be shown similarly.  We have thus verified that $\tbigoplus_{i\in I}L_i$ satisfies property~\ref{proper.F}.
 \end{proof}

By a standard category theory result (see e.g. \cite[Theorem~16.8]{joy}), Propositions~\ref{Fishered} and \ref{productsF} imply the following: 

\begin{corollary}\index{frame!with property \textup{(F)}|textsl}\index{property!F@\textup{(F)}|textsl}
Locales satisfying property~\ref{proper.F}\index{subcategory!epireflective~$\sim$|textsl} are epireflective in the category of locales.
\end{corollary}

Hence, property \ref{proper.F} is a separation property which is well-behaved categorically, which lies between  property \ref{proper.T1} and $\mathcal{F}$-separatedness, and it can be described by a first-order formula. Hence it seems to  deserve some further investigation, and this is what we do in the subsequent subsection.

%

\subsection{Closed points, fitted points and properties (H), (F), and ($T_U$)}

Recall from Proposition~\ref{FimpliesT1} that property \ref{proper.F} implies property \ref{proper.T1}. On the other hand, it is well known that properties \ref{proper.H} or \ref{proper.TU} also imply property \ref{proper.T1} (see \cite[IV.3.3.2]{separation} and \cite[IV.1.3]{separation} respectively). In this context, it is also natural to consider a ``dual'' of \ref{proper.T1} as follows:

\begin{definition}
A frame has \emph{fitted points} (briefly, it satisfies \ref{proper.ptfit}) if for any prime $p\in L$, the sublocale $\bl(p)$ is fitted --- i.e., if
\begin{equation}{\label{proper.ptfit}}
\text{for any prime } p\in L,\ \bl(p)=\tbigcap\set{ \op(a)\mid p\in\op(a)}. \tag*{\textup{(pt-fit)}}
\end{equation}

\end{definition}

\begin{remarks}
\begin{enumerate}[label=\textup{(\roman*)},leftmargin=1.5\parindent]
\item We recall here that a topological space is $T_1$ if and only if each singleton is saturated. In this sense, \ref{proper.ptfit} is also an axiom of $T_1$-type, which plays the role of the counterpart of the localic property \ref{proper.T1} in the closed-fitted duality.
\item For topological spaces, all the singletons being closed is equivalent to all the singletons being saturated. However, in the localic setting it turns out that \ref{proper.ptfit} and \ref{proper.T1} are not comparable. To see this, let $L$ be a pointless locale and denote by $L_{*}$ (resp. $L^{*}$) the locale obtained by adjoining a new bottom (resp. top) element to $L$. It is then readily verified that $L_*$ satisfies \ref{proper.ptfit} but not \ref{proper.T1}; and similarly $L^*$  satisfies \ref{proper.T1} but not \ref{proper.ptfit}.
\end{enumerate}
\end{remarks}

The following result indicates that properties \ref{proper.H}, \ref{proper.F} and \ref{proper.TU} have a similar behaviour w.r.t  $T_1$-locales and w.r.t. locales with fitted points:

\begin{proposition} 
Each of the properties \ref{proper.H}, \ref{proper.F} and  \ref{proper.TU}  implies \linebreak property~\ref{proper.ptfit}.
\end{proposition}

\begin{proof}
Let us first assume that $L$ satisfies \ref{proper.TU}.  Let $p$ be a prime, set $S:=\tbigcap_{p\in\op(a)}\op(a)$, and denote by $f$ be the composite localic map \[\begin{tikzcd}
	S & {\bl(p)} & L
	\arrow[two heads, from=1-1, to=1-2]
	\arrow[hook, from=1-2, to=1-3]
\end{tikzcd}\]
where the first map is the unique surjection onto the terminal locale --- i.e., $f$ is given by $f(a)=p$ for all $a<1$ and $f(1)=1$. Moreover, let $j$ be the embedding of $S$ in $L$. We have $j\leq f$.  Indeed, if $a<1$ in $S$, we have to check that $a\leq p$. By contradiction, if $a\not\leq p$, then we have $a\in S\subseteq \op(a)$, which implies $a=1$, a contradiction. Hence $j\leq f$, and  so by property \ref{proper.TU}  it follows that $j=f$, which in turn implies $S=\bl(p)$.

Assume now that $L$ satisfies property \ref{proper.F}. Let  $p$ be a prime. We need to show that $\tbigcap_{p\in\op(a)}\op(a)\subseteq \bl(p)$. Hence let $b\ne 1$ such that $b\in\op(a)$ whenever $p\in\op(a)$ --- i.e., such that $a\to b\leq b$ whenever $a\not\leq p$. We need to show that $b=p$. Note that $b\leq p$ (if $b\not\leq p$, then $1=b\to b\leq b$, a contradiction). Now, if $p\not\leq b$, by property \ref{proper.F}, there are $u,v\in L$ such that $u\to p\ne p$, $v\to b\ne b$ and $u\vee v=1$. Now, $u\to p\ne p$ means $u\leq p$ and so $p\vee v=1$. In particular $v\not\leq p$. It follows that $v\to b\leq b$, a contradiction.

The fact that property \ref{proper.H} implies \ref{proper.ptfit} can be dealt with similarly.
\end{proof}

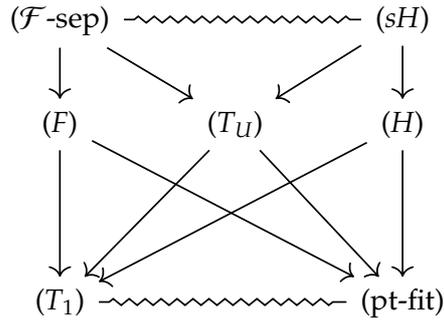
\begin{figure}[h]
\[\begin{tikzcd}
	{(\mathcal{F}\text{-sep})} && {(sH)} \\
	{(F)} & {(T_U)} & {(H)} \\
	{} \\
	{(T_1)} && {\text{(pt-fit)}}
	\arrow[from=1-1, to=2-1]
	\arrow[from=1-3, to=2-3]
	\arrow[from=2-1, to=4-1]
	\arrow[from=2-1, to=4-3]
	\arrow[from=2-3, to=4-3]
	\arrow[from=2-3, to=4-1]
	\arrow[from=1-1, to=2-2]
	\arrow[from=1-3, to=2-2]
	\arrow[from=2-2, to=4-1]
	\arrow[from=2-2, to=4-3]
	\arrow[squiggly, no head, from=1-1, to=1-3]
	\arrow[squiggly, no head, from=4-1, to=4-3]
\end{tikzcd}\]
\caption{Relations between properties discussed in this subsection}
\label{fig.ptfit}
\end{figure}

\subsection{Property (F) is not comparable with other weakenings of fitness}
\label{sec.weakenings}

Since property \ref{proper.F} is weaker than fitness, one may wonder whether it is comparable with any of the other relaxations of fitness that appear in the literature --- notably, subfitness, weak subfitness and prefitness \cite{subfit2015} (see also Subsection~\ref{subsection.standard.sep}). It is not, as the following observations show:
\begin{enumerate}[label       = \textbullet,
                  labelindent = 0\parindent,
                  leftmargin  = \parindent,
                  labelsep    = *,
                  topsep      = 10pt,
                  itemsep      = 5pt]
                 \item Property \ref{proper.F} does not imply any of subfitness, weak subfitness and prefitness.  Indeed, property \ref{proper.F} is hereditary and since hereditary subfitness, weak subfitness and prefitness are all equivalent to fitness (see \cite[Prop.~6]{subfit2015},  \cite[Cor.~1]{subfit2015} and  \cite[Prop.~9]{subfit2015} respectively); we would have that property \ref{proper.F} implies fitness --- but even $\mathcal{F}$-separatedness does not imply fitness (recall Theorem~\ref{fitnessimpliesfsep}).
\item Prefit does not imply property \ref{proper.F}: by Proposition~\ref{FimpliesT1} we would have that prefitness implies \ref{proper.T1}, but this is not true by \cite[Example~5.2]{subfit2015}.
\item Subfitness (and so weak subfitness) does not imply property \ref{proper.F}: If  subfitness implies property \ref{proper.F}, then by Proposition~\ref{FimpliesT1} it also implies property \ref{proper.T1}, but  it is well known this does not hold in general. 
\end{enumerate}

\section{The strong Hausdorff property does not imply property (F)}

%
We begin by a generalization of a result due to Banaschewski (cf. \cite[Prop.~A.6]{singly}) to a context not restricted to simple extensions:

\begin{proposition}\label{st.haus.sep}
Let $X$ be a topological space, $U$ an open subspace and $F$ its closed complement. Suppose that
\begin{enumerate}[label       = \textbullet,
                  labelindent = 0\parindent,
                  leftmargin  = \parindent,
                  labelsep    = *,
                  topsep      = 10pt,
                  itemsep      = 5pt]\item $X$ is Hausdorff,
\item The subspace topology on $U$ is strongly Hausdorff,
\item The subspace topology on $F$ is a spatial multiplier --- i.e., $\Omega(F)\oplus \Omega(Z)\cong \Omega(F\times Z)$ for any space $Z$.
\end{enumerate}
Then the topology of $X$ is strongly Hausdorff.
\end{proposition}

\begin{proof}
Given a frame $L$ and $a\in L$ we recall \cite[Lemma~A.5]{singly} that the maps 
$$\alpha_1\colon L\oplus L\to \newdownarrow a \oplus \newdownarrow a,\qquad \alpha_1(b\oplus c)=(b\wedge a)\oplus (c\wedge a),$$
$$\alpha_2\colon L\oplus L\to \newuparrow a \oplus \newdownarrow a,\qquad \alpha_2(b\oplus c)=(b\vee a)\oplus (c\wedge a),$$
$$\alpha_3\colon L\oplus L\to \newdownarrow a \oplus \newuparrow a,\qquad \alpha_3(b\oplus c)=(b\wedge a)\oplus (c\vee a),$$
$$\alpha_4\colon L\oplus L\to \newuparrow a \oplus \newuparrow a,\qquad \alpha_4(b\oplus c)=(b\vee a)\oplus (c\vee  a)$$
are jointly monic. Hence, if we set $L=\Omega(X)$ and $a=U$, we will prove that $X$ is strongly Hausdorff by showing that for any $V\in\Omega(X)$ the equality  $\alpha_i( d_{\Omega(X)}\vee X\oplus V)= \alpha_i( d_{\Omega(X)}\vee V\oplus X)$   (see \cite[Theorem~III.5.3.3\,(b)]{separation}) for all $i=1,2,3,4$, where
$$d_{\Omega(X)}=\tbigvee\set{U_1\oplus U_2\mid U_1,U_2\in\Omega(X),\ U_1\cap U_2=\varnothing}.$$
For that purpose, we first compute $\alpha_i(d_{\Omega(X)})$ for $i=1,2,3,4$:
\begin{enumerate}[label       = \textbullet,
                  labelindent = 0\parindent,
                  leftmargin  = \parindent,
                  labelsep    = *,
                  topsep      = 10pt,
                  itemsep      = 5pt]\item  We have $\alpha_1(d_{\Omega(X)})= \tbigvee\set{(U_1\cap U)\oplus (U_2\cap U)\mid U_1,U_2\in\Omega(X),\ U_1\cap U_2=\varnothing}=d_{\Omega(U)}$ where the last equality holds because $U$ is open.
\item The map $\beta\colon \newuparrow U\oplus \newdownarrow U\to \Omega(F)\oplus \Omega(U)$ given by $\beta(U_1\oplus U_2)=(F\cap U_1)\oplus U_2$  is an isomorphism, and the map $\gamma\colon \Omega(F)\oplus \Omega(U)\to \Omega (F\times U)$ given by $\gamma(U_1\oplus  U_2)=U_1\times U_2$ is also an isomorphism because $F$ is a spatial multiplier.  Accordingly we see that $\gamma\beta\alpha_2(d_{\Omega(X)})=\tbigcup\set{(U_1\cap F)\times U_2\mid U_1,U_2\in\Omega(X),\ U_1\cap U_2=\varnothing}.$
Now, let $x\in F$ and $y\in U$. Then, since $x\ne y$ and $X$ is Hausdorff, there are  $U_1,U_2\in \Omega(X)$ with $x\in U_1$, $y\in U_2$ and $U_1\cap U_2=\varnothing$. Hence $(x,y)\in \gamma\beta\alpha_2(d_{\Omega(X)})$. It follows that $\gamma\beta\alpha_2(d_{\Omega(X)})= F\times U=1_{\Omega(F\times X)}$, and so $\alpha_2(d_{\Omega(X)})=1$.
\item Symmetrically we see that $\alpha_3(d_{\Omega(X)})=1$.
\item Similarly, we see that the map $\beta\colon \newuparrow U\oplus \newuparrow U\to  \Omega(F\times F)$ given by $\beta(U_1\oplus U_2)=(U_1\cap F)\times (U_2\cap F)$ is an isomorphism. Hence, $\beta \alpha_4(d_{\Omega(X)}) = \tbigcup\set{(U_1\times F)\times (U_2\cap F)\mid U_1,U_2\in\Omega(X),\ U_1\cap U_2=\varnothing}$. Since $X$ is Hausdorff it follows readily that  $\beta \alpha_4(d_{\Omega(X)}) =F\times F- D_F$ where $D_F=\set{(x,x)\mid x\in F}$.
\end{enumerate}
It is now clear from this description that $\alpha_i( d_{\Omega(X)}\vee X\oplus V)= \alpha_i( d_{\Omega(X)}\vee V\oplus X)$  for any $i=1,2,3,4$ (in the first case because $U$ is strongly Hausdorff, the second and third case trivially, and the last case by a direct computation).
\end{proof}

\begin{construction}\label{consj} \textup{(}cf. Juhász, Soukup and Szentmiklóssy \cite[Lemma~2.1]{juhasz}\textup{)}. Let $(X,\tau_X)$ be a Hausdorff topological space, $\{K_n\}_{n\in\N}$ a pairwise disjoint family of nonempty compact subsets of $X$ and let $Y:=X-\tbigcup_{n\in \N}K_n$. 

In the  disjoint union $Z:= Y\bigsqcup \N$, the family 
\begin{equation}\label{basis}\beta=\set{ U\cap Y\mid U\in\tau_X\} \cup  \{ (U\cap Y)\cup \{n\} \mid  U\in\tau_X,\  K_n\subseteq U}\end{equation}
is clearly closed under binary intersections --- hence $\beta$ constitutes a base of a topology $\tau_Z$ in $Z$. We note that \begin{enumerate}[label=\textup{(\roman*)},leftmargin=1.5\parindent]
\item\label{comm1}$Y$ is open in $(Z,\tau_Z)$, 
\item\label{comm2} the subspace topologies of $Y$ in $(X,\tau_X)$ and $(Z,\tau_Z)$ agree, and \item\label{comm3} the subspace topology of $\N$ in $(Z,\tau_Z)$ is discrete.\end{enumerate}
Now it is easily verified that $(Z,\tau_Z)$ is Hausdorff (for instance, if $n\ne m$ in $\N$, since $K_n\cap K_m=\varnothing$ and disjoint compact sets in a Hausdorff space can be separated by disjoint open sets, there exist $U,V\in\tau_X$ with $U\cap V=\varnothing$, $K_n\subseteq U$ and $K_m\subseteq V$. Then $U'=\{n\}\cup (U\cap Y)$ and $V'=\{m\}\cup (V\cap Y)$ are the desired disjoint open sets in $(Z,\tau_Z)$ separating $n$ and $m$. The remaining cases are similar).
\end{construction}

From Proposition~\ref{st.haus.sep}, the fact that the strong Hausdorff property is   hereditary,   and the previous comments \ref{comm1}---\ref{comm3} we immediately have (recall that a discrete topology is a spatial multiplier):

\begin{corollary}\label{cor.str.haus}
Let $(Z,\tau_Z)$ be as in Construction~\ref{consj} and suppose additionally that the topology of $(X,\tau_X)$ is strongly Hausdorff. Then the topology of $(Z,\tau_Z)$ is also strongly Hausdorff.
\end{corollary}

A space is said to be \emph{anti-Urysohn} \cite{juhasz} if for any open sets $U,V\ne\varnothing$, one has $\overline{U}\cap\overline{V}\ne\varnothing$.
The following is proved more generally in \cite{juhasz}:

\begin{proposition}{\label{prop.a.u.}}\textup{(cf. }\cite[Lemma~2.1]{juhasz}\textup{)}
Let $(Z,\tau_Z)$ be as in Construction~\ref{consj}. Suppose additionally that $Y=X-\tbigcup_{n\in\N} K_n$ is dense in $(X,\tau_X)$ and that for every infinite $A\subseteq \N$ the union $\tbigcup_{n\in A} K_n$ is also dense in $(X,\tau_X)$.
Then for every $\varnothing \ne U\in \tau_Z$ there is an $n_U\in\N$ such that $\newuparrow n_U\subseteq \overline{U}^{Z}$. In particular, $(Z,\tau_Z)$ is anti-Urysohn.
\end{proposition}

\begin{proof}
We refer to \eqref{basis} and note that it clearly suffices to show the statement for an open of the form $U\cap Y\ne\varnothing$ with $U\in\tau_X$. Let $A=\set{ n\in\N\mid U\cap K_n\ne\varnothing}$. Then there must be an $n_U\in\N$ such that $\newuparrow n_U\subseteq A$ (for otherwise, for each $n\in\N$ there is an $m_n\geq n$ with $U\cap K_{m_n}=\varnothing$. The set $\{m_n\mid n\in\N\}$ is clearly infinite and this contradicts the density of $\tbigcup_{n\in\N} K_{m_n}$). Now, fix $n\in A$ and let $W\in\tau_X$ such that $K_n\subseteq W$. Then $W\cap U\ne \varnothing$  and since $Y$ is dense, $Y\cap W\cap U\ne \varnothing$. It follows that $n\in \overline{U}^Z$, as required.
\end{proof}

\begin{example}\label{exm.concrete} (cf. \cite[Thm.~2.2]{juhasz}) Let $\{0,1\}$ denote the two-element discrete space and consider the product topology on $X=\{0,1\}^\N$ (i.e., the Cantor cube). Set 
$$K_n=\{ (a_m)_{m\in\N}\in X\mid a_n=1, a_m=0\ \forall m>n\}.$$
Clearly, $\{K_n\}_{n\in\N}$ is a pairwise disjoint family of nonempty closed (hence compact subsets) of $X$, and it is readily seen that they satisfy the conditions of Proposition~\ref{prop.a.u.}. Hence by Proposition~\ref{prop.a.u.}, the resulting $(Z,\tau_Z)$ is Hausdorff and anti-Urysohn. Actually, because the space $X$ is regular (and in particular strongly Hausdorff), one deduces from Corollary~\ref{cor.str.haus} that $(Z,\tau_Z)$ is strongly Hausdorff.
\end{example}

We conservatively extend the terminology from \cite{juhasz} as follows:

\begin{definition}A locale $L$ will be said to be \emph{anti-Urysohn} if for any $a,b\in L$ with $a,b\ne 0$, one has $a^*\vee b^*\ne 1$. \end{definition}

We recall that a locale is \emph{irreducible} if  $B_L=\{0,1\}$, where $B_L$ denotes the \emph{Booleanization} of $L$ --- i.e., $B_L=\set{ a^*\mid a\in L}=\set{ a\in L\mid a=a^{**}}$ (see \cite{irreducibility.themba}). Before we proceed, we need the following lemma:

\begin{lemma}\label{irreducibility.lemma}
A non-trivial locale with property~\ref{proper.F} cannot be irreducible.
\end{lemma}

\begin{proof}
Let $L$ be a non-trivial irreducible locale with property~\ref{proper.F}. Since $L$ is non-trivial, pick an $a\ne 0,1$ in $L$. By property~\ref{proper.F} there are $u,v\in L$ with $u\vee v=1$,   $u\to a\not\leq a$ and $v^*\ne 0$. By irreducibility, necessarily we have $v^*=1$ --- i.e.,  $v=0$. But then $u=1$, and so $a=u\to a\not\leq a$, a contradiction.
\end{proof}

Finally, the following result explains the failure of  anti-Urysohn locales to have property~\ref{proper.F}.

\begin{proposition}\label{ur.prop.F}
A non-trivial anti-Urysohn locale cannot have property \ref{proper.F}.
\end{proposition}

\begin{proof} Let $L$ be a non-trivial anti-Urysohn locale  and suppose it has property~\ref{proper.F}. By Lemma~\ref{irreducibility.lemma}, it is not irreducible. Hence, select an $a\in B_L$ with $0<a<1$. Since $1\ne a\not\leq a^*$, by an application of property \ref{proper.F} there are $u,v\in L$ with $(u\to a)\vee (v\to a^*)=1$, $u\not\leq a$ and $v\not\leq a^*$. 
Now, one has $u\to a=u\to a^{**}=(u\wedge a^*)^*$ and $v\to a^{*}=(v\wedge a)^*$. Since $u\not\leq a$ and $v\not\leq a^*$, we have $u\wedge a^*\ne 0$ and $v\wedge a\ne 0$, this clearly contradicts the anti-Urysohn property.
\end{proof}

In this situation, from Corollary~\ref{cor.str.haus},  Proposition~\ref{prop.a.u.} and Proposition~\ref{ur.prop.F} we deduce the following.

\begin{proposition}
Let $(Z,\tau_Z)$ be as in Construction~\ref{consj} and suppose that the conditions in Proposition~\ref{prop.a.u.} are satisfied. Assume moreover that $(X,\tau_X)$ is strongly Hausdorff.  Then  $(Z,\tau_Z)$ is strongly Hausdorff and does not satisfy property \ref{proper.F}.
\end{proposition}

In particular, with regard to Example~\ref{exm.concrete}, we have the following

\begin{corollary}\label{main.corollary}
There exist strongly Hausdorff spatial locales which do not have property \ref{proper.F}. In particular, they are not $\FF$-separated, and so they are neither fit.
\end{corollary}

\section{A summary of the implications}\label{summarlocl}

In this final section we summarize our results  in the context of standard localic  $T_1$-type and $T_2$-type separation properties. The following diagram comprises our new separation axioms  together with the usual ones, and the relations between them.

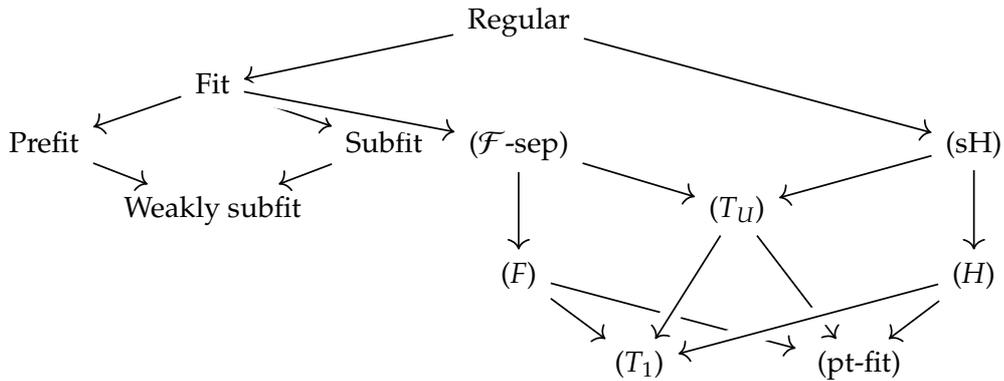
\begin{figure}[h]
\begin{tikzcd}[sep=tiny,arrows={crossing over}]
	&&& {\text{Regular}} \\
	& {\text{Fit}} \\
	{\text{Prefit}} && {\text{Subfit}} & {\text{(}\mathcal{F}\text{-sep)}} &&&& {\text{(sH)}} \\
	& {\text{Weakly subfit}} &&&& {(T_U)}&& \\
	&&& {(F)}& && &{(H)}\\
	 &&&&&&& \\
	 	 &&&&&&& \\
	&&&& \text{(}T_1\text{)}  &&\text{  (pt-fit)}
	\arrow[from=1-4, to=3-8]
	\arrow[from=2-2, to=3-1]
	\arrow[from=2-2, to=3-3]
	\arrow[from=3-1, to=4-2]
	\arrow[from=3-3, to=4-2]
	\arrow[from=1-4, to=2-2]
	\arrow[from=2-2, to=3-4]
	\arrow[from=3-4, to=4-6]
	\arrow[from=3-8, to=4-6]
	\arrow[from=3-4, to=5-4]
	\arrow[from=3-8, to=5-8]
	\arrow[from=4-6, to=8-7]
	\arrow[from=5-4, to=8-7]
	\arrow[from=5-8, to=8-7]
	\arrow[from=4-6, to=8-5]
	\arrow[from=5-4, to=8-5]
	\arrow[from=5-8, to=8-5]
\end{tikzcd}
\caption{Summarizing localic separation}

\label{figtotal}
\end{figure}

All the implications are strict, and the only  implications that hold among the properties above are those which follow from concatenating the implications depicted in the diagram (this follows easily by \cite[Thm.~6.4]{APP21}, Figure~\ref{fig.ptfit}, Subsection~\ref{sec.weakenings}, Remark~\ref{fitness.F} and Corollary~\ref{main.corollary}, together with several facts already known --- see \cite{subfit2015,separation}).

The duality between the strong Hausdorff property and $\FF$-separatedness is apparent in this diagram. Within this parallel, total unorderedness appears to play a symmetric role with respect to both facets of the duality.\\[2mm]
\textbf{Acknowledgement.} I am grateful to Achim Jung, Jorge Picado and Aleš Pultr for useful discussions on the topic discussed in this paper.

\end{document}